\documentclass[12pt,a4paper,twoside,reqno]{amsart}

\usepackage[all,cmtip]{xy}
\usepackage{amsmath,amscd,mathscinet}
\usepackage{amssymb,amsfonts}
\usepackage[driver=pdftex,margin=3cm,heightrounded=true,centering]{geometry}
\usepackage{JK}
\usepackage[colorlinks=true,linkcolor=blue,citecolor=blue]{hyperref}
\usepackage{verbatim}

\usepackage{graphicx}

\title{The Podle\'s sphere as a spectral metric space}
\author{Konrad Aguilar and Jens Kaad}

\address{School of Mathematical and Statistical Sciences, 
Arizona State University, 
901 S. Palm Walk, Tempe, AZ 85287-1804}

\email{konrad.aguilar@asu.edu}

\address{Department of Mathematics and Computer Science,
The University of Southern Denmark,
Campusvej 55, DK-5230 Odense M,
Denmark}

\email{jenskaad@hotmail.com}

\subjclass[2010]{58B32,58B34;46L89,46L30} 
\keywords{Quantum metric spaces, Lip-Norms, Podle\'s sphere, spectral triples}

\begin{document}

\begin{abstract}
We study the spectral metric aspects of the standard Podle\'s sphere, which is a homogeneous space for quantum $SU(2)$. The point of departure is the real equivariant spectral triple investigated by D\polhk{a}browski and Sitarz. The Dirac operator of this spectral triple interprets the standard Podle\'s sphere as a $0$-dimensional space and is therefore not isospectral to the Dirac operator on the $2$-sphere. We show that the seminorm coming from commutators with this Dirac operator provides the Podle\'s sphere with the structure of a compact quantum metric space in the sense of Rieffel.
%
\end{abstract}

\maketitle

\tableofcontents

\section{Introduction}
In the work of Podle\'{s} a two parameter family of deformations of the $2$-sphere is introduced and interpreted as homogeneous spaces for Woronowicz' quantum group versions of the Lie group $SU(2)$, \cite{Podles87,Woro87}. Since these Podle\'s spheres are deformations of a $2$-sphere attention has been given to the interpretation of these objects within the framework of noncommutative differential geometry, \cite{DaSi03,DALW07,NeTu05,ScWa04,ReSe14,KrWa13}. An important aspect of this investigation is to understand the topology on the state space of the Podle\'s sphere in question. Following the program of Rieffel on  compact quantum  metric spaces, candidates for metrics on the state space come from seminorms on Podle\'s sphere and the question is whether the topology on the state space coming from this kind of metrics agrees with the weak $*$-topology, \cite{Rieffel98a}. Interesting examples of seminorms arise from commutators with an unbounded operator acting on a Hilbert space and this is what constitutes the link between Rieffel's ideas and Connes' program on noncommutative differential geometry, \cite{Con96, Connes}. Indeed, the unbounded operator in question might very well be the unbounded operator of a spectral triple, \cite{Connes89, Connes}, see also \cite{Bellissard10} where this kind of compact quantum metric spaces are called spectral metric spaces.

We are in this paper investigating the seminorm coming from a particular spectral triple over the standard Podle\'s sphere $S_q^2$. This spectral triple was analyzed in detail by D\polhk{a}browski and Sitarz and it has the property of being real and equivariant with respect to the coaction of quantum $SU(2)$, \cite{DaSi03}. The unbounded operator appearing in this spectral triple does not have the same spectrum as the Dirac operator on the $2$-sphere, indeed the corresponding spectral dimension of the standard Podle\'s sphere is $0$. In some sense this is the correct spectral dimension since the standard Podle\'s sphere (for $q \in (0,1)$) is isomorphic to the unitization of the compact operators on a separable Hilbert space.

Our main result is that the seminorm $L$ coming from the D\polhk{a}browski-Sitarz spectral triple over the standard Podle\'s sphere provides a metric on the state space and that this metrizes the weak $*$-topology. In other words, we prove that $(S_q^2,L)$ is a compact quantum metric space in the sense of Rieffel. We emphasize that there are plenty of explicit seminorms on $S_q^2$ that yields the right topology on the state space, indeed since we are just working with the unitization of the compacts, it suffices to ensure that the elements in the corresponding unit ball are of sufficiently rapid decay. An important point is that our seminorm comes from a particular unbounded operator, which is interesting from the point of view of noncommutative differential geometry.

A natural question to address in the future is whether the standard Podle\'s spheres converge to the $2$-sphere in the quantum Gromov-Hausdorff propinquity as introduced by Latr\'emoli\`ere, \cite{Latremoliere13}. Our work is a necessary first step in this direction. The convergence question for the standard Podle\'s spheres is complementary to Rieffel's work on matrix algebras converging to the $2$-sphere, \cite{Rieffel01, Rieffel15}. Indeed, the fuzzy spheres also form a family of Podle\'s spheres but for a different deformation parameter.

We end this introduction by giving a few comments on our methods. The approach that we are following is motivated by the geometry of the classical $2$-sphere. The circle action given by rotation of the $2$-sphere around the axis from the south pole to the north pole survives under the $q$-deformation. The corresponding fixed point algebra is a commutative unital $C^*$-subalgebra of $S_q^2$ and it is given by the continuous functions on the subspace $\{q^{2k} \mid k \in \nn \cup \{0\} \} \cup \{0\} \su \cc$, which one may naturally think of as a quantized version of the unit interval $[0,1]$. In the classical case, thus when $q = 1$, the fiber over each point $t \in [0,1]$ is the unit circle except at $0$ and $1$ where the fiber reduces to a single point. This picture can also be transferred to the $q$-deformed case, but the fiber over each point is now (due to the discrete nature of the base space) a separable Hilbert space. Interestingly enough, the singularity at $t = 1$ in the classical case blows up and we also see an infinite dimensional Hilbert space as the fiber over the point $q^0 = 1$. 

Our approach consists of analyzing each of these fibers separately, paying particular attention to the interplay between the Hilbert space norm and the seminorm $L$. In order to then conclude anything about the behaviour of the whole standard Podle\'s sphere we prove a version of the fundamental theorem of analysis for the derivation $d_q = [D_q, \cd]$ coming from the D\polhk{a}browski-Sitarz spectral triple. To explain what we mean by this, we let $\psi_\infty : S_q^2 \to \cc$ denote the state which projects onto the subspace generated by the unit. We then construct a bounded operator
\[
\int : X \to S_q^2
\]
(on a suitable operator space $X$) such that
\[
\int d_q(x) = x - \psi_\infty(x) \cd 1_{S_q^2}
\]
for all elements $x$ in the Lipschitz algebra associated to the D\polhk{a}browski-Sitarz spectral triple.

Remark that, as a byproduct of our results, we obtain a metric on each of the quantized unit intervals $\{q^{2k} \mid k \in \nn \cup \{0\}\}$, $q \in (0,1)$, and it would be interesting to compute this metric explicitly.

\subsection*{Acknowledgements}
%

We gratefully acknowledge the Syddansk Universitet Odense, Arizona State University, and the Institute of Mathematics of the Polish Academy of Sciences (IMPAN) for facilitating this collaboration. Both authors were partially sponsored by the grant H2020-MSCA-RISE-2015-691246-QUANTUM DYNAMICS and Polish Ministry of Science and Higher Education grant Polish Government grant \#3542/H2020/2016/2 during their visit at IMPAN in January 2018.

This project was initiated during the Great Plains Operator Theory Symposium (GPOTS) at the Texas Christian University in May 2017.
%
%
%

The second author was partially supported by the DFF-Research Project 2 ``Automorphisms and Invariants of Operator Algebras'', no. 7014-00145B and by the Villum Foundation (grant 7423).

\section{Preliminaries on spectral metric spaces}\label{s:prelim}

\subsection{Derivations and seminorms from unbounded operators}\label{ss:derivation}
Throughout this subsection we fix a densely defined unbounded \emph{selfadjoint} operator $D : \T{Dom}(D) \to H$ acting on a separable Hilbert space $H$, together with a $C^*$-algebra $A$ represented on $H$ via a $*$-homomorphism $\pi : A \to \B L(H)$. Usually $\pi$ will be faithful but this is not a requirement in the following. Both the norm on $A$ and on $H$ will be denoted by $\| \cd \|$, but the operator norm will be denoted by $\| \cd \|_\infty$.

\begin{dfn}\label{d:l-algebra}
The \emph{Lipschitz algebra} (associated to $(\pi,D)$) is the $*$-subalgebra $\T{Lip}_D(A) \subseteq A$ defined by
\[
\begin{split}
\T{Lip}_D(A) := 
\big\{ x \in A 
& \mid \pi(x)\big( \T{Dom}(D) \big) \su \T{Dom}(D) \, , \, \, [D, \pi(x)] : \T{Dom}(D) \to H \\
& \QQ \T{ has a bounded extension to }H \big\}  .
\end{split}
\]
For $x \in \T{Lip}_D(A)$ we denote the bounded extension of $[D,\pi(x)] : \T{Dom}(D) \to H$ by $d(x) \in \B L(H)$.
\end{dfn}

We remark that the assignment $d : \T{Lip}_D(A) \to \B L(H)$ is a closed derivation, satisfying the relation
\[
d(x^*) = - d(x)^* \Q x \in \T{Lip}_D(A)  .
\]

\begin{dfn}\label{d:l}
Define $L_D : A \to [0,\infty]$ by the formula
\[
L_D(x) := \sup_{ \xi,\eta \in \T{Dom}(D) \, , \, \, \| \xi \|, \| \eta \| = 1} \big| \inn{ \xi, \pi(x^*) D \eta} - \inn{\pi(x) D \xi, \eta} \big|
\]
\end{dfn}

We record the following fundamental result, which is nonetheless a bit difficult to find anywhere:

\begin{lemma}\label{l:domain}
Let $x \in A$. We have that
\[
\big( L_D(x) < \infty \big) \lrar \big( x \in \T{Lip}_D(A) \big)  .
\]
In this case 
\begin{equation}\label{eq:semi}
L_D(x) = \| d(x) \|_\infty  .
\end{equation}
\end{lemma}
\begin{proof}
Suppose that $x \in \T{Lip}_D(A)$. Then we have that
\[
\begin{split}
L_D(x) & = \sup_{ \xi,\eta \in \T{Dom}(D) \, , \, \, \| \xi \|, \| \eta \| = 1} \big| \inn{ \xi, \pi(x^*) D \eta} - \inn{\pi(x) D \xi, \eta} \big| \\
& = \sup_{ \xi,\eta \in \T{Dom}(D) \, , \, \, \| \xi \|, \| \eta \| = 1} \big| \binn{ [D,\pi(x)] \xi, \eta} \big| \\
& = \| d(x) \|_\infty  ,
\end{split}
\]
proving the implication ``$\lar$'' and the identity in Equation \eqref{eq:semi}.

Suppose that $L_D(x) < \infty$. For each $\xi, \eta \in \T{Dom}(D)$, define the complex number 
\[
\varphi_{\xi,\eta}(x) := \inn{ \xi, \pi(x^*) D \eta} - \inn{\pi(x) D \xi, \eta}  .
\]
Remark that $| \varphi_{\xi,\eta}(x) | \leq L_D(x) \cd \| \xi \| \cd \| \eta \|$. We thus have that
\[
\big| \inn{\pi(x) \xi, D \eta} \big| = \big| \varphi_{\xi,\eta}(x) + \inn{\pi(x) D \xi, \eta} \big|
\leq L_D(x) \cd \| \xi \| \cd \| \eta \| + \| x \| \cd \| D \xi \| \cd \| \eta \|  ,
\]
for all $\xi, \eta \in \T{Dom}(D)$. But this implies that $\pi(x) \xi \in \T{Dom}(D^*)$ for all $\xi \in \T{Dom}(D)$ and hence by selfadjointness of $D$ that $\pi(x)\big( \T{Dom}(D) \big) \subseteq \T{Dom}(D)$. The rest of the proof should now be clear.
\end{proof}

The following result is immediate:

\begin{lemma}\label{l:Leibniz}
The assignment $L_D : A \to [0,\infty]$ is lower semi-continuous and defines a seminorm on $\T{Lip}_D(A)$. Moreover, $L_D$ is $*$-invariant and Leibniz in the sense that:
\[
\begin{split}
L_D(x^*) = L_D(x) \Q \mbox{and}\Q
L_D(x y) \leq L_D(x) \cd \| y \| + \| x \| \cd L_D(y) 
\end{split}
\]
for all $x,y \in A$.
\end{lemma}

\subsection{Spectral metric spaces}
We recall the main definition in the theory of compact quantum metric spaces:

\begin{dfn}[{\cite{Rieffel98a, Rieffel99, Rieffel05}}]\label{d:Lcqms}
Let $A$ be a unital $C^*$-algebra and let $L: A \rightarrow [0,\infty]$ be a seminorm.
We call the pair $(A, L)$  a {\em compact quantum metric space} and $L$ a {\em Lip-norm} if the following hold:
\begin{enumerate}
\item the domain, $\T{dom}(L):=\{ a \in A \mid L(a) < \infty \}$ is dense in $A$,
\item  $L$ is $*$-invariant and lower semi-continuous on $A$,
\item the kernel $\T{ker} L:=\{ a \in A \mid L(a)=0\}=\cc1_A$, where $1_A$ is the unit of $A$,
\item the {\em Monge-Kantorovich metric} defined on the state space of $A$ by
\begin{equation*}
mk_L (\mu,\nu):=\sup \{ |\mu(a)-\nu(a)| \mid a \in A, L(a) \leq1 \} \Q \T{ for all } \mu, \nu \in \T{S}(A)
\end{equation*}
metrizes the weak $*$-topology.
\end{enumerate}
If furthermore $L$ is a {\em Leibniz seminorm}, i.e. for any $a,b \in A$, we have
\begin{equation*}
 L(ab)  \leq L(a) \cd \|b\| + \|a\| \cd L(b),
\end{equation*}
 then we call
$(A, L)$  a {\em Leibniz compact quantum metric space} and $L$ a {\em Leibniz Lip-norm}.
\end{dfn}

The usual way of proving that a pair $(A,L)$ consisting of a unital $C^*$-algebra and a seminorm is a compact quantum metric space is by applying the following theorem due to Rieffel:
\begin{thm}[{\cite[Theorem 1.8 and Theorem 1.9]{Rieffel98a}}, {\cite[Proposition 1.3]{Ozawa05}}]\label{t:Rieffel}
Let $A$ be a unital $C^*$-algebra and let $L:A \rightarrow [0, \infty]$ be a lower semi-continuous $*$-invariant  seminorm.

If $\T{dom}(L)$ is dense in $A$, and  $\T{ker} L=\cc1_A$, then the following conditions are equivalent:
\begin{enumerate}
\item the pair $(A,L)$ is a  compact quantum metric space; 
\item there exists a state $\mu \in \T{S}(A)$ such that the set: 
\begin{equation*}
\{ a\in \T{dom}(L) \mid L (a) \leq 1 \text{ and } \mu(a) = 0 \}
\end{equation*}
is totally bounded for the $C^*$-norm on $A$.
\end{enumerate}
\end{thm}

Building on Connes' initial work on metrics on state spaces coming from spectral triples \cite{Connes89, Connes}, Bellissard, Marcolli, and Reihani formalized the notion of a {\em compact spectral metric space} in \cite{Bellissard10}. We recall their definition here in a version adapted to the context of the present paper.

\begin{dfn}[{\cite[Definition 1]{Bellissard10}}]\label{d:csms}
Let $(\C A,H,D)$ be a spectral triple for a unital $C^*$-algebra $A$ and assume that the representation $\pi: A \rightarrow \B L(H)$ is unital and faithful. If $(A,L_D)$ is a  compact quantum metric space we call it a \emph{compact spectral metric space}.
\end{dfn}

Remark that when the seminorm $L_D : A \to [0,\infty]$ comes from a unital spectral triple (as in the above definition) it is automatic that $L_D$ is $*$-invariant, Leibniz and lower-semicontinuous, see Subsection \ref{ss:derivation}. It can also easily be seen that $\cc \cdot 1_A \subseteq \T{ker}L_D$. Thus, in order for $(A,L_D)$ to be a compact spectral metric space  all that remains to check is the condition  $\cc \cdot 1_A \supseteq \T{ker}L_D$ and (2) of Theorem  \ref{t:Rieffel}.

\section{The $q$-deformed spectral triple on the Podle\'{s} sphere}\label{s:spectral}
In this section we introduce the D\polhk{a}browski-Sitarz spectral triple on the standard Podle\'s sphere. We are following the conventions for quantum $SU(2)$ and the standard Podle\'s sphere applied by D\polhk{a}browski-Sitarz in \cite{DaSi03}. However, our approach follows the approach taken by Neshveyev and Tuset where the relationship between the Dirac operator and the quantized universal enveloping algebra is clarified, \cite{NeTu05}.

We fix $q \in (0,1)$. We let $SU_q(2)$ denote the universal unital $C^*$-algebra with two generators $a$ and $b$, subject to the relations:
\[
\begin{split}
& ba = q  ab \Q b^* a = q a b^* \Q bb^* = b^* b \\
& \Q a^* a + q^2 b b^* = 1 = aa^* + bb^*  .
\end{split}  
\]
The unital $*$-subalgebra generated by $a,b \in SU_q(2)$ is denoted by $\C O(SU_q(2))$.

We let $\pa_k : \C O(SU_q(2)) \to \C O(SU_q(2))$ denote the automorphism defined by
\[
\begin{array}{ll}
\pa_k(a) = q^{1/2} a & \Q \pa_k(b) = q^{1/2} b \\
\pa_k(a^*) = q^{-1/2} a^* & \Q \pa_k(b^*) = q^{-1/2} b^*  .
\end{array}
\]
The inverse automorphism is denoted by $\pa_{k^{-1}} : \C O(SU_q(2)) \to \C O(SU_q(2))$.

Moreover, we have the linear maps
\[
\pa_e \, \, , \, \, \, \pa_f : \C O(SU_q(2)) \to \C O(SU_q(2))
\]
defined by
\[
\begin{split}
\pa_e(a) = - b^* \Q \pa_e(b) = q^{-1} \cd a^* \Q \pa_e(a^*) = 0 = \pa_e(b^*) \\
\pa_f(a^*) = qb \Q \pa_f(b^*) = - a \Q \pa_f(a) = 0 = \pa_f(b) = 0 
\end{split}
\]
and the relations:
\[
\begin{split}
\pa_e(x \cd y) & = \pa_e(x) \cd \pa_k(y) + \pa_{k^{-1}}(x) \cd \pa_e(y) \\
\pa_f(x \cd y) & = \pa_f(x) \cd \pa_k(y) + \pa_{k^{-1}}(x) \cd \pa_f(y)  ,
\end{split}
\]
for all $x,y \in \C O(SU_q(2))$.

We let $\C O(S_q^2)$ denote the unital $*$-subalgebra of $SU_q(2)$ generated by the elements 
\[
A := b^* b \Q \T{and} \Q B := ab^* \in SU_q(2)
\]
We record the relations:
\[
\begin{array}{ll}
AB = q^2 B A &  A = A^* \\
B B^* = q^{-2} A(1-A) & B^*B = A(1 - q^2 A)  .
\end{array}
\]
The standard Podle\'s quantum sphere is the unital $C^*$-algebra $S_q^2$ obtained as the norm-closure of $\C O(S_q^2)$ inside $SU_q(2)$.

For each $n \in \zz$, we let $X_n$ denote the $C^*$-correspondence from $S_q^2$ to $S_q^2$ defined as the norm-closure of
\[
\C A_n := \big\{ x \in \C O(SU_q(2))  \mid \pa_k(x) = q^{n/2} x \big\}
\]
inside $SU_q(2)$, the inner product being given by $\inn{x,x} := x^* x$ for all $x \in \C A_n$. We remark that $\C A_0 = \C O(S_q^2)$.

The derivations $\pa_1 := q^{1/2}\pa_e : \C O(S_q^2) \to X_{-2}$ and $\pa_2 := q^{-1/2}\pa_f : \C O(S_q^2) \to X_2$ are then given on generators by
\begin{equation}\label{eq:derval}
\begin{array}{lll}
\pa_1(A) = b^* a^*  & \, \, \pa_1(B) = -(b^*)^2 & \, \, \pa_1(B^*) = q^{-1} (a^*)^2 \\
\pa_2(A) = - ab & \, \, \pa_2(B) = -q^{-1} a^2 & \, \, \pa_2(B^*) = b^2  .
\end{array}
\end{equation}
We remark that these two derivations $\pa_1$ and $\pa_2$ are closable and related by the formula
\[
\pa_1(x)^* = - \pa_2(x^*) \Q x \in \C O(S_q^2)  .
\]

We let $h : SU_q(2) \to \cc$ denote the Haar state and define the Hilbert spaces $H_h$ as the completion of $SU_q(2)$  with respect to the inner product
\[
\inn{x,y} := h(x^* y) \Q x,y \in SU_q(2)  .
\]
We let $H_+$ and $H_- \su H_h$ denote the Hilbert spaces obtained as the closures of $\C A_1$ and $\C A_{-1}$, respectively. For details on the Haar state (and quantum $SU(2)$ in general) see \cite{KlSc97}.

We define the symmetric unbounded operator
\[
\sD_q := \ma{cc}{ 0 & \pa_f \\ \pa_e & 0 } : \C A_1 \op \C A_{-1} \to H_+ \op H_-  ,
\]
and we let $D_q : \T{Dom}(D_q) \to H_+ \op H_-$ denote the closure of $\sD_q$. The left action of $\C O(S_q^2) = \C A_0$ on $\C A_1 \op \C A_{-1}$ induced by the multiplication in $\C O(SU_q(2))$ then yields a faithful representation 
\[
\pi = \ma{cc}{\pi_+ & 0 \\ 0 & \pi_-} : S_q^2 \to H_+ \op H_-  .
\]

The following result is from \cite{DaSi03} (in fact in \cite{DaSi03} also the uniqueness, reality and equivariance are addressed):

\begin{thm}[D\polhk{a}browski-Sitarz]\label{t:dabsit}
The triple $\big( \C O(S_q^2), H_+ \op H_-, D_q\big)$ is an even spectral triple.
\end{thm}

We let $\T{Lip}_{D_q}(S_q^2) \su S_q^2$ denote the Lipschitz algebra associated to the Dirac operator $D_q : \T{Dom}(D_q) \to H_+ \op H_-$ and record the formula
\[
d_q(x) = \ma{cc}{0 & \pa_2(x) \\ \pa_1(x) & 0} \Q x \in \C O(S_q^2)
\]
for the associated derivation. In particular we obtain extensions of $\pa_1$ and $\pa_2$ to the Lipschitz algebra $\T{Lip}_{D_q}(S_q^2)$. These two extensions are again denoted by $\pa_1$ and $\pa_2$:
\[
\pa_1 : \T{Lip}_{D_q}(S_q^2) \to \B L(H_+, H_-) \Q \pa_2 : \T{Lip}_{D_q}(S_q^2) \to \B L(H_-, H_+)  .
\]
We do not know whether the images of $\pa_1$ and $\pa_2$ are contained in the $C^*$-algebra $SU_q(2)$ even though this holds for their restrictions to the coordinate algebra $\C O(S_q^2)$.

\section{Matrix units for the standard Podle\'s sphere}
We fix $q \in (0,1)$. In this section we discuss the well-known result that the unital $C^*$-algebra $S_q^2$ is isomorphic to the unitization of the compact operators on a separable Hilbert space, see \cite[Proposition 4]{Podles87}: 

\begin{prop}[Podle\'s]\label{p:comisom}
The standard Podle\'s sphere $S_q^2$ is $*$-isomorphic to the unitization of the compact operators on the separable Hilbert space $\ell^2(\nn \cup \{0\})$ via the unital representation
\[
\pi : S_q^2 \to \B L( \ell^2(\nn \cup \{0\})) \Q \pi(A)(e_k) := q^{2k} e_k \, , \, \, \pi(B)(e_k) = q^k \sqrt{1 - q^{2(k+1)}} e_{k+1},
\]
where $\{e_k\}_{k = 0}^\infty$ denotes the standard orthonormal basis for $\ell^2(\nn \cup \{0\})$.
\end{prop}

For later use, we record that
\[
\pi(B^*)(e_k) = q^{k-1} \sqrt{1 - q^{2k}} e_{k - 1} \Q k > 0 \, , \, \, \pi(B^*)(e_0) = 0.
\]

We notice that the spectrum of the positive element $A \in S_q^2$ is given by
\[
\T{Sp}(A) = \{0\} \cup \{ q^{2k} \mid k \in \nn \cup \{0\} \} .
\]
For $k \in \nn \cup \{0\}$, we let $\chi_{\{q^{2k}\}} : \T{Sp}(A) \to \{0,1\}$ denote the indicator function for the subset $\{q^{2k}\} \subseteq \T{Sp}(A)$. Remark that $\chi_{\{q^{2k}\}}$ is in fact continuous.
\medskip

We apply the concrete representation from Proposition \ref{p:comisom} to find explicit matrix units $f_{n,k} \in S_q^2$, $n,k \in \nn \cup \{0\}$. These matrix units will play an important role in the following sections.
\medskip

For each $n,k \in \nn \cup\{0\}$, we define the constants
\begin{equation}\label{eq:normal}
C_{n,k} := \fork{ccc}{ \prod_{j = n}^{k-1} q^{2j} (1 - q^{2(j+1)}) & \T{for} & 0 \leq n < k \\ 
1 & \T{for} & n = k \\
\prod_{j = k}^{n-1} q^{2j} (1 - q^{2(j+1)}) & \T{for} & n > k },
\end{equation}
and then the elements
\begin{equation}\label{eq:onb}
f_{n,k} := \fork{ccc}{
\frac{1}{\sqrt{C_{n,k}}} B^{n-k} \cd \chi_{\{q^{2k}\}}(A) & \T{for} & n \geq k \\
\frac{1}{\sqrt{C_{n,k}}} (B^*)^{k - n} \cd \chi_{\{q^{2k}\}}(A) & \T{for} & n \in \{0,1,\ldots,k-1\} 
} 
\end{equation}
in $S_q^2$. These elements are the matrix units for the standard Podle\'s sphere:

\begin{lemma}\label{l:matrix}
We have the relations
\[
\begin{split}
f_{n,k}^* = f_{k,n} \Q f_{n,k} \cd f_{m,l} = \fork{ccc}{ f_{n,l} & \T{for} & k = m \\ 0 & \T{for} & k \neq m} \\
\end{split}
\]
for all $n,k,m,l \in \nn \cup \{0\}$.
\end{lemma}
\begin{proof}
Let $n,k \in \nn \cup \{0\}$. By Proposition \ref{p:comisom} it suffices to show that 
\[
\pi(f_{n,k})(e_l) = \fork{ccc}{e_n & \T{for} & k = l \\ 0 & \T{for} & k \neq l}.
\]
The case where $k \neq l$ is immediate. For $k = l$ and $n > k$, we compute that
\[
\pi(f_{n,k})(e_k) = \frac{1}{\sqrt{C_{n,k}}} \pi(B^{n-k}) (e_k)
= \frac{1}{\sqrt{C_{n,k}}} \cd \prod_{j = k}^{n-1} q^j \sqrt{1 - q^{2(j+1)}} \cd e_n
= e_n.
\]
Similarly, for $k = l$ and $n < k$, we compute that
\[
\pi(f_{n,k})(e_k) = \frac{1}{\sqrt{C_{n,k}}} \pi(B^*)^{k-n} (e_k)
= \frac{1}{\sqrt{C_{n,k}}} \cd \prod_{j = n}^{k-1} q^j \sqrt{1 - q^{2(j+1)}} \cd e_n
= e_n.
\]
The case where $k = l$ and $n = k$ is also immediate and the lemma is proved.
\end{proof}



\section{Derivatives of matrix units}
We are interested in computing the value of the derivation
\[
\pa_1 : \T{Lip}_{D_q}(S_q^2) \to \B L(H_+,H_-)
\]
when applied to the matrix units $f_{n,k} \in S_q^2$, $n,k \in \nn \cup \{0\}$. We start out softly by showing that the projections $\chi_{\{q^{2k}\}}(A)$, $k \in \nn \cup \{0\}$, lie in the Lipschitz algebra $\T{Lip}_{D_q}(S_q^2)$ and we compute the value of the derivation $\pa_1$ on these particular elements.
\medskip

We record the commutation relations:
\begin{equation}\label{eq:bBrel}
\begin{split}
b^* B & = q B b^* \Q a^* B^*  = q B^* a^* \\
a^* A & = q^2 A a^* \Q b^* A = A b^*,
\end{split}
\end{equation}
which follow directly from the defining relations for $SU_q(2)$.

\begin{lemma}\label{l:derpotens}
It holds for each $n \in \nn \cup \{0\}$ that
\[
\begin{split}
\pa_1(A^n) & = \frac{1 - q^{2n}}{1 - q^2} A^{n-1} b^* a^*  \Q \pa_1(B^n) = - \frac{1 - q^{2n}}{1 - q^2} B^{n-1} (b^*)^2 \\
& \Q \pa_1( (B^*)^n) = q^{-1} \frac{1 - q^{2n}}{1 - q^2} (B^*)^{n-1} (a^*)^2  .
\end{split}
\]
\end{lemma}
\begin{proof}
This follows from the Leibniz rule for 
\[
\pa_1 : \T{Lip}_{D_q}(S_q^2) \to \B L(H_+,H_-),
\]
the relations in Equation \eqref{eq:bBrel} and the fact that $\pa_1(A) = b^* a^*$, $\pa_1(B) = - (b^*)^2$ and $\pa_1(B^*) = q^{-1} (a^*)^2$ (see Equation \eqref{eq:derval}).
\end{proof}

\begin{lemma}\label{l:indicator}
Let $k \in \nn \cup \{0\}$. Then
\[
\chi_{\{q^{2k}\}}(A) = \lim_{n \to \infty}\Big( \frac{1}{q^{2kn}} A^n - \sum_{j = 1}^k \frac{1}{q^{2jn}} \chi_{ \{ q^{2(k-j)}\}}(A) \Big)  ,
\]
where the convergence takes place in the norm on $S_q^2$.
\end{lemma}
\begin{proof}
This is true since $\T{Sp}(A) = \{0\} \cup \{ q^{2k} \mid k \in \nn \cup \{0\}\}$.
\end{proof}

We now present our formula for the derivatives of indicator functions associated to the isolated points in the spectrum of $A$:

\begin{lemma}\label{l:derind}
Let $k \in \nn \cup \{0\}$. It holds that $\chi_{\{q^{2k}\}}(A) \in \T{Lip}_{D_q}(S_q^2)$ and the derivative is given by
\[
\begin{split}
\pa_1\big( \chi_{\{q^{2k}\}}(A) \big) 
& = \frac{1}{q^{2k} (1 - q^2)} \chi_{ \{q^{2k}\}}(A) \cd b^* a^* \\
& \Q - \frac{1}{q^{2(k-1)} (1 - q^2)} \chi_{ \{q^{2(k-1)}\}}(A) \cd b^* a^* . 
\end{split}
\]
\end{lemma}
\begin{proof}
The proof runs by strong induction on $k \in \nn \cup \{0\}$.

For each $n \in \nn$:
\begin{equation*}
\pa_1 (A^n) =\frac{1 - q^{2n}}{1 - q^2} A^{n-1} b^* a^*  , 
\end{equation*}
which converges to $\frac{1}{1-q^2}\cd \chi_{\{1\}}(A)\cdot b^* a^*$ by Lemma \ref{l:indicator}. By taking adjoints and using the relation $\pa_2(A^n) = - \pa_1(A^n)^*$ we also obtain that $\pa_2(A^n)$ converges to $-\frac{1}{1-q^2}\cd a b \cd \chi_{\{1\}}(A)$. Using that the derivation $d_q : \T{Lip}_{D_q}(S_q^2) \to \B L(H_+ \op H_-)$ is closed, this proves the statement for $k = 0$.
%

Now, let $k\in \nn$ and suppose that our statement is true for all $j \in \{0,1,\ldots,k - 1\}$. For each $n \geq 2$ we then compute that
\begin{align*}
 & \pa_1\left(\frac{1}{q^{2k n}} A^n - \sum_{j = 1}^k \frac{1}{q^{2jn}} \chi_{ \{ q^{2(k-j)}\}}(A) \right) \\
 & \Q =  \frac{1}{q^{2k n}}\cd  \frac{1 - q^{2n}}{1 - q^2} A^{n-1} \cd b^* a^* 
 - \sum_{j = 1}^{k} \frac{1}{q^{2jn}} \frac{1}{q^{2(k-j)} (1 - q^2)} \chi_{ \{q^{2(k-j)}\}}(A) \cd b^* a^*  \\
& \QQ - \sum_{j = 1}^{k} \frac{1}{q^{2(k-j-1)} (1 - q^2)} \chi_{ \{q^{2(k-j-1)}\}}(A)  \cd b^* a^* \\
 & \Q = \frac{1}{q^{2k}(1-q^2)} \bigg(\frac{1}{q^{2k(n-1)}}A^{n-1}  - \sum_{j = 1}^{k} \frac{1}{q^{2j(n-1)}} \chi_{\{q^{2(k - j)}\}}(A) 
 \bigg) \cd b^* a^* \\
& \QQ - \frac{1}{q^{2(k - 1)}(1 - q^2)} \bigg( \frac{1}{q^{2(k - 1)(n-1)}} A^{n-1} - \sum_{j = 1}^{k - 1} \frac{1}{q^{2j(n-1)}} 
\chi_{ \{q^{2(k-j-1)}\}}(A) \bigg) \cd b^* a^*  ,\\
\end{align*}
which, again by Lemma \ref{l:indicator}, converges to
\[
\frac{1}{q^{2k} (1 - q^2)} \chi_{ \{q^{2k}\}}(A) \cd b^* a^*
- \frac{1}{q^{2(k-1)} (1 - q^2)} \chi_{ \{q^{2(k-1)}\}}(A) \cd b^* a^* .
\]
By taking adjoints we then obtain that
\[
\pa_2\left(\frac{1}{q^{2k n}} A^n - \sum_{j = 1}^{k} \frac{1}{q^{2jn}} \chi_{ \{ q^{2(k-j)}\}}(A) \right)
\]
converges to
\[
\frac{1}{q^{2(k-1)} (1 - q^2)} a b \cd \chi_{ \{q^{2(k-1)}\}}(A)  - \frac{1}{q^{2k} (1 - q^2)} a b \cd \chi_{ \{q^{2k}\}}(A) .
\]
Using again that the derivation $d_q : \T{Lip}_{D_q}(S_q^2) \to \B L(H_+ \op H_-)$ is closed, this proves the lemma.
\end{proof}

\subsection{Commmutator relations for indicator functions}
Before continuing with our computation of the derivation
\[
\pa_1 : \T{Lip}_{D_q}(S_q^2) \to \B L(H_+, H_-),
\]
we need to understand the relationship between the spectral projections $\chi_{\{q^{2k}\}}(A)$, $k \in \nn \cup \{0\}$, and the element $a^* \in SU_q(2)$.
%
%

\begin{lemma}\label{l:dec-comm}
Let $k \in \nn \cup \{0\}$. It holds that
\[
\chi_{\{q^{2k}\}}(A) \cd a^* = a^* \cd \chi_{\{q^{2(k+1)}\}}(A)  .
\]
\end{lemma}
\begin{proof}
First, consider $k=0$. Notice that $a^*\chi_{\{1\}}(A)=0$ since $a a^* = 1 - A$. We gather for all $n \in \nn$ that
\[
A^n a^*  
= a^* \frac{A^n}{q^{2n}} 
= a^* \frac{A^n}{q^{2n}} - a^*\frac{\chi_{\{1\}}(A)}{q^{2n}}
= a^* (\frac{A^n}{q^{2n}} -\frac{\chi_{\{1\}}(A)}{q^{2n}})  .
\]
Since $\chi_{\{1\}}(A) = \lim_{n \to \infty} A^n$ and $\chi_{\{q^2\}}(A) = \lim_{n \to \infty}(\frac{A^n}{q^{2n}} -\frac{\chi_{\{1\}}(A)}{q^{2n}})$, by Lemma \ref{l:indicator}, we have that $\chi_{\{1\}}(A) a^* = a^* \chi_{\{q^2\}}(A)$.

Next, fix $k \in \nn$ and assume the statement of the lemma is true for all $0 \leq m \leq k$. We gather for all $n \in \nn$ that
\begin{equation*}
\begin{split}
& \left(\frac{1}{q^{2(k+1)n}}A^n - \sum_{j=1}^{k+1} \frac{1}{q^{2jn}}\chi_{\{q^{2(k+1-j)}\}}(A) \right) a^*\\
& \Q =a^* \left( \frac{1}{q^{2(k+2)n}}A^n - \sum_{j=1}^{k+1} \frac{1}{q^{2jn}}\chi_{\{q^{2(k+2-j)}\}}(A) \right) \\
& \Q = a^* \left( \frac{1}{q^{2(k+2)n}}A^n - \sum_{j=1}^{k+1} \frac{1}{q^{2jn}}\chi_{\{q^{2(k+2-j)}\}}(A) \right) - a^*\frac{1}{q^{2(k+2)n}}\chi_{\{1\}}(A)\\
& \Q =a^* \left( \frac{1}{q^{2(k+2)n}}A^n - \sum_{j=1}^{k+2} \frac{1}{q^{2jn}}\chi_{\{q^{2(k+2-j)}\}}(A) \right)  ,
\end{split}
\end{equation*}
and thus taking limits, using Lemma \ref{l:indicator}, we have proved the result by strong induction.
\end{proof}

\subsection{Derivatives of matrix units}\label{ss:derbasis}
We are now ready to compute the value of the derivation 
\[
\pa_1 : \T{Lip}_{D_q}(S_q^2) \to \B L(H_+,H_-)
\]
when applied to the matrix units $f_{n,k} \in S_q^2$, $n,k \in \nn \cup \{0\}$, see Equation \eqref{eq:onb} for the definition. We start by computing these values without considering the normalizing constants $C_{n,k} \in (0,\infty)$.
%

\begin{lemma}\label{l:derbasI}
Let $k \in \nn \cup \{0\}$. We have the identity
\begin{equation}\label{eq:ngeq1}
\begin{split}
& (1 - q^2) \cd \pa_1\big( B^n \chi_{ \{q^{2k}\}}(A) \big) \\
& \Q = B^{n-1} \cd \Big( (q^{-2(k+1)} - 1) \cd  \chi_{ \{q^{2(k+1)}\}}(A)
+ (q^{2n} - q^{-2k}) \cd \chi_{ \{q^{2k}\}}(A) \Big) \cd (b^*)^2
\end{split}
\end{equation}
for all $n \geq 1$ and the identity
\begin{equation}\label{eq:nleqk}
\begin{split}
& (1 - q^2) \cd \pa_1\big( (B^*)^n \chi_{ \{q^{2k}\}}(A) \big) \\
& \Q = (B^*)^{n+1} \cd \Big( q^{-4k-2} \cd \chi_{ \{q^{2(k+1)}\}}(A) - q^{2(n+1 -2k)}  \cd \chi_{ \{q^{2k}\}}(A) \Big) \cd (b^*)^2
\end{split}
\end{equation}
for all $n \in \{0,\ldots,k\}$.
\end{lemma}
\begin{proof}
We start by taking care of the case where $n = 0$ in Equation \eqref{eq:nleqk}. It follows from Lemma \ref{l:derind} and Lemma \ref{l:dec-comm} that
\begin{equation}\label{eq:casezero}
\begin{split}
& (1 - q^2) \cd \pa_1( \chi_{ \{q^{2k}\}}(A) ) 
= q^{-2k} b^* a^* \chi_{\{q^{2(k+1)}\}}(A) - q^{-2(k-1)} b^* a^* \chi_{\{q^{2k}\}}(A) \\
&\Q = q^{-2k} b^* a^*(q^{-2(k+1)} A \chi_{\{q^{2(k+1)}\}}(A) )- q^{-2(k-1)} b^* a^*(q^{-2k}A \chi_{\{q^{2k}\}}(A)) \\
& \Q = q^{-4k - 2} B^* \chi_{\{q^{2(k+1)}\}}(A) (b^*)^2  - q^{-4k + 2}  B^*  \chi_{\{q^{2k}\}}(A) (b^*)^2 .
\end{split}
\end{equation}

We continue by considering Equation \eqref{eq:ngeq1}. For $n \geq 1$, we apply Equation \eqref{eq:casezero} together with Lemma \ref{l:derpotens} to compute that
\[
\begin{split}
& (1 - q^2) \cd \pa_1( B^n \chi_{\{q^{2k}\}}(A)) \\
& \Q = - (1 - q^{2n}) \cd  B^{n-1} \chi_{\{q^{2k}\}}(A) (b^*)^2 
+ q^{-2(k+1)} (1 - q^{2(k+1)}) B^{n-1}\chi_{\{q^{2(k+1)}\}}(A) (b^*)^2 \\ 
& \QQ - q^{-2k} (1 - q^{2k})   B^{n-1} \chi_{\{q^{2k}\}}(A) (b^*)^2 \\
& \Q = ( q^{2n} - q^{-2k} ) B^{n-1} \chi_{\{q^{2k}\}}(A) (b^*)^2
+ (q^{-2(k+1)} - 1)  B^{n-1} \chi_{\{q^{2(k+1)}\}}(A) (b^*)^2  .
\end{split}
\]

Finally, we study the remaining part of Equation \eqref{eq:nleqk}. For $n \in \{1,\ldots,k\}$, we again apply Equation \eqref{eq:casezero} and Lemma \ref{l:derpotens} to obtain that
\[
\begin{split}
& (1 - q^2) \cd \pa_1( (B^*)^n \chi_{\{q^{2k}\}}(A)) \\
& \Q = q^{-1} (1 - q^{2n}) \cd  (B^*)^{n-1} (a^*)^2 \chi_{\{q^{2k}\}}(A)
+ q^{-4k-2} (B^*)^{n+1} \chi_{\{q^{2(k+1)}\}}(A) (b^*)^2 \\ 
& \QQ - q^{-4k + 2} (B^*)^{n+1} \chi_{\{q^{2k}\}}(A) (b^*)^2 \\
& \Q = -q^{-4k + 2 + 2n} (B^*)^{n+1} \chi_{\{q^{2k}\}}(A) (b^*)^2 
+ q^{-4k-2} (B^*)^{n+1} \chi_{\{q^{2(k+1)}\}}(A) (b^*)^2,
\end{split}
\]
where we have also used that $(a^*)^2 \chi_{\{q^{2k}\}}(A) = q^{-4k + 3} (B^*)^2 \chi_{\{q^{2k}\}}(A) (b^*)^2$. This proves the lemma. 
\end{proof}

The next lemma provides the desired formula for the value of the derivation $\pa_1$ on the matrix units:

\begin{lemma}\label{l:derbasII}
Let $k,n \in \nn \cup \{0\}$. We have the identity
\[
\begin{split}
(1 - q^2) \cd \pa_1(f_{n,k}) & = q^{-3k-2} \sqrt{1 - q^{2(k+1)}} \cd f_{n,k+1} \cd (b^*)^2 \\
& \Q - q^{-2k - n + 1} \sqrt{1 - q^{2n}} \cd f_{n-1,k} \cd (b^*)^2 ,
\end{split}
\]
where we set $f_{-1,k}:=0$ for all $k \in \nn \cup \{0\}$.
\end{lemma}
\begin{proof}
We recall the definition of the normalizing constants $C_{n,k} \in (0,\infty)$ from Equation \eqref{eq:normal}.

For $n > k$ we apply Lemma \ref{l:derbasI} and the fact that 
\[
\frac{C_{n,k+1}}{ C_{n,k}} = \frac{1}{q^{2k}(1 - q^{2(k+1)})} \Q \T{and} \Q
\frac{C_{n-1,k}}{C_{n,k}} = \frac{1}{q^{2(n-1)}(1 - q^{2n})},
\]
to compute that
\[
\begin{split}
(1 - q^2) \cd \pa_1(f_{n,k}) & = \frac{C_{n,k+1}^{1/2}}{C_{n,k}^{1/2}} \cd (q^{-2(k+1)} -1) \cd f_{n,k+1} \cd (b^*)^2 \\
& \Q + \frac{C_{n-1,k}^{1/2}}{C_{n,k}^{1/2} } \cd (q^{2(n-k)} - q^{-2k}) \cd f_{n-1,k} \cd (b^*)^2 \\
& = q^{-3k - 2} \sqrt{1 - q^{2(k+1)}} \cd f_{n,k+1} \cd (b^*)^2 \\
& \Q - q^{-2k -n + 1} \sqrt{1 - q^{2n}} \cd  f_{n-1,k} \cd (b^*)^2 .
\end{split}
\]

For $0 \leq n \leq k$ we apply Lemma \ref{l:derbasI} one more time together with the identities
\[
\frac{C_{n,k+1}}{C_{n,k}} = q^{2k}(1 - q^{2(k+1)}) \Q \T{and} \Q
\frac{C_{n-1,k}}{C_{n,k}} = q^{2(n-1)}(1 - q^{2n})
\]
to obtain that
\[
\begin{split}
(1 - q^2) \cd \pa_1(f_{n,k}) & = \frac{C_{n,k+1}^{1/2}}{C_{n,k}^{1/2}} q^{-4k-2} \cd f_{n,k+1} \cd (b^*)^2 \\
& \Q - \frac{C_{n-1,k}^{1/2}}{C_{n,k}^{1/2}} q^{2(-n + 1 - k)} \cd f_{n-1,k} \cd (b^*)^2 \\
& = q^{-3k-2} \sqrt{ 1 - q^{2(k+1)}} \cd f_{n,k+1} \cd (b^*)^2 \\
& \Q - q^{-n + 1 - 2k} \sqrt{ 1 - q^{2n}} \cd f_{n-1,k} \cd (b^*)^2.
\end{split}
\]
This proves the present lemma.
\end{proof}

\section{Totally bounded subspaces of the fibers}
Let us fix a $k \in \nn \cup \{0\}$. We define the subspace
\[
Y_k :=  S_q^2 \cd \chi_{\{q^{2k}\}}(A) \su S_q^2
\]
and notice that $Y_k$ is automatically closed in the norm on $S_q^2$ since $\chi_{\{q^{2k}\}}(A) \in S_q^2$ is a projection. We think of $Y_k \su S_q^2$ as the fiber over the point $q^{2k} \in \T{Sp}(A)$.

We are going to study each of these closed subspaces, paying particular attention to the relationship between the linear map
\[
\pa_1 : \T{Lip}_{D_q}(S_q^2) \cap Y_k  \to \B L(H_+,H_-)
\]
and the norm on $Y_k$. We aim for the following:

\begin{thm}\label{t:0-totbou}
Let $k \in \nn \cup \{0\}$. The subspace
\[
\big\{ x \in \T{Lip}_{D_q}(S_q^2) \cap Y_k 
\mid \| \pa_1(x) \|_\infty  \leq 1 \big\} \su Y_k
\]
is totally bounded.
\end{thm}


The following structural result is a consequence of Proposition \ref{p:comisom} and Lemma \ref{l:matrix}:

\begin{prop}\label{p:fibhilb}
The closed subspaces $Y_k \su S_q^2$ and $Y_k \cd (b^*)^2 \su SU_q(2)$ are both isometrically isomorphic to the separable Hilbert space $\ell^2(\nn \cup \{0\})$. Under this identification, the sequence $\{ f_{n,k} \}_{n = 0}^\infty$ in $Y_k$ and the sequence $\{ f_{n,k} (b^*)^2 q^{-2k} \}$ in $Y_k \cd (b^*)^2$ are orthonormal bases.
\end{prop}

We are going to present a computation of our derivation on general elements in the subspace $\T{Lip}_{D_q}(S_q^2) \cap Y_k \su Y_k$. This computation is slightly tricky because we a priory do not know much about the range of the derivation $\pa_1 : \T{Lip}_{D_q}(S_q^2) \to \B L(H_+,H_-)$. Indeed, the Lipschitz algebra $\T{Lip}_{D_q}(S_q^2)$ could be bigger than the $*$-algebra obtained by taking the closure of the derivation $d_q : \C O(S_q^2) \to SU_q(2)$.

\begin{lemma}\label{l:vertder}
Let $k \in \nn \cup \{0\}$ and $x = \sum_{m = 0}^\infty \la_m \cd f_{m,k} \in \T{Lip}_{D_q}(S_q^2) \cap Y_k$ be given. We have the identities:
\[
\begin{split}
& (1-q^2) \cd f_{k,n} \cd \pa_1(x) \cd \chi_{\{q^{2k}\}}(A) \\
& \Q = - q^{-2k-n}\sqrt{1 - q^{2(n+1)}}\cd \la_{n+1} \cd \chi_{\{q^{2k}\}}(A) (b^*)^2 \Q \mbox{and} \\
& (1 - q^2) \cd f_{k+1,n} \cd \pa_1(x) \cd \chi_{\{q^{2(k+1)}\}}(A) \\
& \Q = q^{-3k-2}\sqrt{1 - q^{2(k+1)}} \cd \la_n \cd \chi_{\{q^{2(k+1)}\}}(A) (b^*)^2  ,
\end{split}
\]
for all $n \in \nn \cup \{0\}$.
\end{lemma}
\begin{proof}
Let $n \in \nn \cup \{0\}$ be given. Using Lemma \ref{l:derbasII} and Lemma \ref{l:matrix} we have that
\[
\begin{split}
& (1-q^2) \cd f_{k,n} \cd \pa_1(x) \cd \chi_{\{q^{2k}\}}(A)  \\
& \Q = 
- (1 - q^2) \cd \chi_{\{q^{2k}\}}(A) \cd \pa_1( f_{k,n}) \cd x \\
& \QQ + (1 - q^2) \cd \chi_{\{q^{2k}\}}(A) \cd \pa_1( \la_n \cd \chi_{\{q^{2k}\}}(A)) \cd \chi_{\{q^{2k}\}}(A) \\
& \Q = 
- q^{-3n-2} \sqrt{1 - q^{2(n+1)}} \cd f_{k,n+1} \cd (b^*)^2 \cd x \\
& \Q = - q^{-2k -n} \sqrt{1 - q^{2(n+1)}} \cd  \la_{n+1} \cd \chi_{\{q^{2k}\}}(A) (b^*)^2,
\end{split}
\]
where we have also used that $(b^*)^2 f_{m,k} = q^{2(m - k)} f_{m,k} (b^*)^2$ for all $m \in \nn \cup \{0\}$. This proves the first of the two identities.

The second of the two identities follows from the computation:
\[
\begin{split}
& (1 - q^2) \cd f_{k+1,n} \cd \pa_1(x) \cd \chi_{\{q^{2(k+1)}\}}(A)  \\
& \Q = 
(1 - q^2) \cd  \pa_1( \la_n \cd f_{k+1,k}) \cd \chi_{\{q^{2(k+1)}\}}(A) \\
& \Q = q^{-3k-2}\sqrt{1 - q^{2(k+1)}} \cd \la_n \cd \chi_{\{q^{2(k+1)}\}} (b^*)^2 ,
\end{split}
\]
where we have used Lemma \ref{l:derbasII} and Lemma \ref{l:matrix} one more time.
\end{proof}

We now present our computation of the derivatives of Lipschitz elements in the fiber $Y_k$. 
%

\begin{lemma}\label{p:vertder}
Let $k \in \nn \cup \{0\}$. The restriction of the derivation $\pa_1 : \T{Lip}_{D_q}(S_q^2) \to \B L(H_+, H_-)$ to the subspace $\T{Lip}_{D_q}(S_q^2) \cap Y_k \subseteq \T{Lip}_{D_q}(S_q^2)$ takes values in the subspace $(Y_k + Y_{k + 1}) \cd (b^*)^2 \subseteq S_q^2 \cd (b^*)^2$. Moreover, for $x = \sum_{n = 0}^\infty \la_n \cd f_{n,k} \in \T{Lip}_{D_q}(S_q^2) \cap Y_k$, we have an explicit formula for the derivative:
\[
\begin{split}
(1 - q^2) \cd \pa_1(x) & =
\sum_{n = 0}^\infty q^{-3k-2} \sqrt{1 - q^{2(k+1)}} \cd \la_n \cd f_{n,k+1} (b^*)^2 \\
& \Q - \sum_{n = 0}^\infty q^{-2k - n + 1}\sqrt{1 - q^{2n}} \cd \la_n \cd f_{n-1,k} (b^*)^2 .
\end{split}
\]
\end{lemma}
\begin{proof}
For each $N \in \nn \cup \{0\}$, it follows by Lemma \ref{l:vertder} that
\[
\sum_{n = 0}^N f_{n,k} f_{n,k}^* \cd (1 - q^2) \cd \pa_1(x) \cd \chi_{\{q^{2k}\}}(A) 
= - \sum_{n = 0}^N q^{-2k-n}\sqrt{1 - q^{2(n+1)}} \cd \la_{n+1} \cd f_{n,k} (b^*)^2 .
\]
Since both $\sum_{n = 0}^N f_{n,k} f_{n,k}^*$ and $\chi_{\{q^{2k}\}}(A) \in S_q^2$ are orthogonal projections and hence have operator norm bounded by one, we obtain that
\[
\sum_{n = 0}^N | q^{-n}\sqrt{ 1 - q^{2(n+1)}} \cd \la_{n + 1} |^2 \leq  (1 - q^2)^2 \cd \| \pa_1(x) \|_\infty^2  ,
\]
for all $N \in \nn \cup \{0\}$. But this shows that 
\[
\begin{split}
& -\sum_{n = 0}^\infty q^{-2k-n + 1}\sqrt{1 - q^{2n}} \cd \la_n \cd f_{n-1,k} (b^*)^2 \\
& \Q = 
\sum_{n = 0}^\infty (1 - q^2) \cd \pa_1( \la_n \cd f_{n,k} ) \cd \chi_{\{q^{2k}\}}(A)
\in Y_k \cd (b^*)^2  .
\end{split}
\]
A similar argument proves that
\[
\begin{split}
& \sum_{n = 0}^\infty q^{-3k-2} \sqrt{1 - q^{2(k+1)}} \cd \la_n \cd f_{n,k+1} (b^*)^2 \\
& \Q = 
\sum_{n = 0}^\infty (1 - q^2) \cd \pa_1(\la_n \cd f_{n,k}) \cd \chi_{\{q^{2(k+1)}\}}(A) 
\in Y_{k+1} \cd (b^*)^2 .
\end{split}
\]
Using that $\pa_1 : \T{Lip}_{D_q}(S_q^2) \to \B L(H_+,H_-)$ is closable we may conclude that 
\[
\begin{split}
(1 - q^2) \cd \pa_1(x) 
& = \lim_{N \to \infty} \sum_{n = 0}^N (1 - q^2) \cd \pa_1( \la_n \cd f_{n,k})  \\
& = \sum_{n = 0}^\infty q^{-3k-2} \sqrt{1 - q^{2(k+1)}} \cd \la_n \cd f_{n,k+1} (b^*)^2 \\
& \Q - \sum_{n = 0}^\infty q^{-2k - n + 1}\sqrt{1 - q^{2n}} \cd \la_n \cd f_{n-1,k} (b^*)^2 \\
& \in (Y_k + Y_{k+1}) \cd (b^*)^2 .
\end{split}
\]
This proves the lemma.
\end{proof}

The first consequence of our computation of derivatives is a norm-bound on the elements in the Lip-ball of $\T{Lip}_{D_q}(S_q^2) \cap Y_k$:

\begin{lemma}\label{l:norbou}
Let $k \in \nn \cup \{0\}$. The subspace
\[
\big\{ x \in \T{Lip}_{D_q}(S_q^2) \cap Y_k \mid \| \pa_1(x) \|_\infty \leq 1 \big\} \subseteq Y_k
\]
is norm-bounded by $q^k$. 
%
\end{lemma}
\begin{proof}
Let $x = \sum_{n = 0}^\infty \la_n \cd f_{n,k} \in \T{Lip}_{D_q}(S_q^2) \cap Y_k$ and suppose that $\| \pa_1(x)\|_\infty \leq 1$. We have from Lemma \ref{p:vertder} that
\[ 
(1 - q^2) \cd \pa_1(x) \cd \chi_{\{q^{2(k+1)}\}}(A) = \sum_{n = 0}^\infty q^{-3k-2} \sqrt{1 - q^{2(k+1)}} \cd \la_n \cd f_{n,k+1} (b^*)^2 
\]
But this implies that
\[
\begin{split}
q^{-2k} (1 - q^{2(k+1)}) \cd \| x \|^2 & = q^{-2k} (1 - q^{2(k+1)}) \sum_{n = 0}^\infty |\la_n|^2 \\ 
& = \| (1 - q^2) \cd \pa_1(x) \cd \chi_{\{q^{2(k+1)}\}}(A) \|_\infty^2 \leq (1 - q^2)^2 ,
\end{split}
\]
and hence that $\| x \| \leq q^k (1 - q^2) (1 - q^{2(k+1)})^{-1/2} \leq q^k$. This proves the lemma.
\end{proof}

The second consequence of our computation of derivatives is our first main result:

\begin{proof}[Proof of Theorem \ref{t:0-totbou}]
Let $x = \sum_{n= 0}^\infty \lambda_n \cd f_{n,k} \in L_k$ be given. We are going to show that
\[
|\la_n| \leq q^{n - 1}
\]
for all $n \in \nn \cup \{0\}$ (independently of $x$). Since the sequence $\{q^{n-1}\}_{n = 0}^\infty$ is $2$-summable this implies that $L_k  \subseteq Y_k$ is totally bounded by \cite[Theorem 5.5.6]{Shirali06}.

By Lemma \ref{p:vertder} we have that
\begin{equation*}
(1 - q^2) \cd \pa_1(x) \cd \chi_{\{q^{2k}\}}(A)= - \sum_{n= 0}^\infty q^{-2k-n+1} \sqrt{1 - q^{2n}} \cd \la_n \cd f_{n-1,k} (b^*)^2 ,
\end{equation*}
and hence that
\[
\sum_{n = 1}^\infty  q^{-2n + 2} (1 -q^{2n}) \cd |\la_n|^2  = (1 - q^2)^2 \| \pa_1(x) \cd \chi_{\{q^{2k}\}}(A) \|_\infty^2 \leq (1 - q^2)^2 .
\]
We conclude that 
\[
| \la_n| \leq q^{n-1} \cd \frac{1 - q^2}{\sqrt{1 - q^{2n}}}\leq q^{n-1}  .
\]
for all $n \in \nn$. Also, the fact that $|\la_0| \leq q^{-1}$ follows from the estimate in Lemma \ref{l:norbou}. This proves the theorem.
%
\end{proof}

We end this section by presenting an improvement of the computation of derivatives in Lemma \ref{p:vertder}:

\begin{prop}\label{p:totaldern}
Let $x = \mu \cd 1_{S_q^2} + \sum_{l = 0}^\infty \sum_{n = 0}^\infty \la_{n,l} \cd f_{n,l} \in \T{Lip}_{D_q}(S_q^2)$ be given. For each $k \in \nn \cup \{0\}$, it holds that
\[
\begin{split}
& (1-q^2) \cd \pa_1(x) \cd \chi_{\{q^{2k}\}}(A) \\
& \Q = \sum_{n = 0}^\infty \Big( \la_{n,k-1} \cd q^{-3k+1} \sqrt{1 - q^{2k}} - \la_{n+1,k} \cd q^{-2k -n} \sqrt{1 - q^{2(n+1)}} \Big) \cd f_{n,k} (b^*)^2 ,
\end{split}
\]
where by convention $\la_{n,-1} := 0$ for all $n \in \nn \cup \{0\}$. In particular, it holds that $\pa_1(x) \cd \chi_{\{q^{2k}\}}(A) \in Y_k$ for all $k \in \nn \cup \{0\}$. 
\end{prop}
\begin{proof}
Using Lemma \ref{p:vertder} we compute that
\[
\begin{split}
& (1-q^2) \cd \pa_1(x) \cd \chi_{\{q^{2k}\}}(A) \\
& \Q = (1 - q^2) \cd \pa_1( x \cd \chi_{\{q^{2k}\}}(A)) \cd \chi_{\{q^{2k}\}}(A) - (1 - q^2) \cd x \cd \pa_1(\chi_{\{q^{2k}\}}(A)) \cd \chi_{\{q^{2k}\}}(A) \\
& \Q = -\sum_{n = 0}^\infty q^{-2k - n+1} \sqrt{1 - q^{2n}} \cd \la_{n,k} f_{n-1,k}(b^*)^2
+ x \cd q^{-3k + 1} \sqrt{1 - q^{2k}} \cd f_{k-1,k} (b^*)^2 \\
& \Q = -\sum_{n = 0}^\infty q^{-2k - n+1} \sqrt{1 - q^{2n}} \cd \la_{n,k} f_{n-1,k}(b^*)^2 \\
& \Q \Q + \sum_{n = 0}^\infty q^{-3k + 1} \sqrt{1 - q^{2k}} \cd \la_{n,k-1} f_{n,k} (b^*)^2.
\end{split}
\]
This proves the proposition.
\end{proof}

\section{The quantum integral}
In this section we introduce our main device, which will enjoy similar properties to the classical Volterra operator. For this reason we call this device the quantum integral. The quantum integral will be the sum of a ``vertical'' and a ``horizontal'' component, which we refer to as the vertical and the horizontal quantum integral. We have chosen the terminology ``vertical'' and ``horizontal'' in order to separate these two quite different operations but do not attach any deeper geometric meaning to this factorization of the total quantum integral.

The domain of the quantum integral is a closed subspace of operators
\[
X \subseteq \B L(H_+,H_-) 
\]
defined by the requirement
\[
X := \big\{ \xi \in \B L(H_+,H_-) \mid
\xi \cd \chi_{\{q^{2k}\}}(A) \in  Y_k \cd (b^*)^2 \, \, \T{for all } \, \,  k \in \nn \cup \{0\} \big\}  .
\]

The quantum integral will be a bounded operator
\[
\int : X \to S_q^2
\]
and it will satisfy the following properties:


\begin{thm}\label{t:fundam}
We have the identity
\[
\int \pa_1(x) = x - \psi_\infty(x) \cd 1_{S_q^2} \Q \mbox{for all } x \in \T{Lip}_{D_q}(S_q^2) .
\]
\end{thm}

\begin{prop}\label{p:continuity}
We have the estimate
\[
\| ( \int \xi )\cd \chi_{[0,q^{2k}]}(A)  \| \leq \| \xi \|_\infty \cd q^k \cd (k + 2) \cd (1 - q)^{-2}  ,
\]
for all $\xi \in X$ and all $k \in \nn \cup \{0\}$.
\end{prop}

The definition of the quantum integral and the proofs of these results will be given in Subsection \ref{ss:proofs}.


%

\subsection{The quantum vertical integral}
In this subsection we construct the vertical part of the quantum integral. We start by introducing some convenient isometries and a bounded diagonal operator.

For each $k \in \nn \cup \{0\}$, denote the bounded vertical and horizontal shift  operators given by
\begin{equation}\label{eq:isometries}
\begin{split}
& S^V : Y_k \to Y_k \Q S^V(f_{n,k}) := f_{n + 1,k} \Q \T{and} \\
& S^H : Y_k \to Y_{k + 1} \Q S^H(f_{n,k}) := f_{n+1,k+1} .
\end{split}
\end{equation}
Moreover, for each $k \in \nn \cup \{0\}$, let $\Ga : Y_k \to Y_k$ denote the bounded diagonal operator given by
\[
\begin{split}
\Ga( f_{n,k}) := 
\fork{ccc}{ q^{n-k-1} \left( \frac{1 - q^{2(k+1)}}{1 - q^{2n}} \right)^{1/2} f_{n,k} & \T{for} & n \geq k+1 \\
0 & \T{for} & 0 \leq  n \leq k}  .
\end{split}
\]
We remark that $\Ga$ has operator norm equal to one,  $\| \Ga \|_\infty = 1$.

\begin{dfn} We define the \emph{quantum vertical integral}
\[
\int^V : X  \to S_q^2
\]
by the formula
\[
\int^V \xi
:=  - \sum_{ m = 0 }^\infty q^m \frac{1 - q^2}{( 1 - q^{2(m+1)})^{1/2}} \cd \sum_{l = 0}^m \Ga (S^H \Ga)^{m - l} S^V \big(  \xi \cd \chi_{\{q^{2l}\}}(A)\cd b^2 \cd  q^{-2l} \big)  .
\]
\end{dfn}

\begin{lemma}
The quantum vertical integral is a well-defined bounded operator and the operator norm is bounded by $(1 - q^2)^{1/2}(1-q)^{-2}$.
\end{lemma}
\begin{proof}
Let $\xi \in X$. The result follows from the estimate:
\[
\begin{split}
& \sum_{ m = 0 }^\infty q^m \frac{1 - q^2}{( 1 - q^{2(m+1)})^{1/2}} \sum_{l = 0}^m \big\| \Ga (S^H \Ga)^{m - l} S^V \big( \xi \cd \chi_{\{q^{2l}\}}(A) \cd b^2 \cd q^{-2l}\big) \big\| \\
& \Q \leq \sum_{m = 0}^\infty (1 - q^2)^{1/2} q^m \cd (m+1) \cd \| \xi \|_\infty = (1 - q^2)^{1/2} (1 - q)^{-2} \cd \| \xi \|_\infty  . \qedhere
\end{split}
\]
\end{proof}

The next lemma illustrates how the vertical integral is going to form one part of the total quantum integral:

\begin{lemma}\label{l:quavertf}
Let $x = \mu \cd 1_{S_q^2} + \sum_{k = 0}^\infty \sum_{n =0}^\infty \la_{n,k} \cd f_{n,k} \in \T{Lip}_{D_q}(S_q^2)$ be given. We have the identity
\[
\int^V \pa_1(x) = \sum_{m= 0}^\infty \sum_{n = m+1}^\infty \la_{n,m} \cd f_{n,m}  .
\]
\end{lemma}
\begin{proof}
Using Proposition \ref{p:totaldern} we have that
\[
\begin{split}
& (1-q^2) \pa_1(x) \cd  \chi_{\{q^{2l}\}}(A)\cd b^2\cd  q^{-2l} \\
& \Q = \sum_{n = 0}^\infty \la_{n,l-1} \cd q^{-l+1} \sqrt{1 - q^{2l}} - \la_{n+1,l} \cd q^{ -n} \sqrt{1 - q^{2(n+1)}} \Big) \cd f_{n,l}.
\end{split}
\]
We therefore obtain that
\begin{equation}\label{eq:out-vder}
\begin{split}
& \int^V \pa_1( x)  \\
& \Q =- \sum_{ m = 0 }^\infty q^m \frac{1 - q^2}{( 1 - q^{2(m+1)})^{1/2}} \cd \sum_{l = 0}^m \Ga (S^H \Ga)^{m - l} S^V \big(  \pa_1(x) \cd  \chi_{\{q^{2l}\}}(A)\cd b^2\cd  q^{-2l} \big) \\
& \Q =- \sum_{ m = 0 }^\infty \frac{q^m}{( 1 - q^{2(m+1)})^{1/2}} \\ 
& \qqq \cd \sum_{l = 0}^m \sum_{n = l}^\infty 
\Big( \la_{n,l-1} \cd q^{-l+1} \sqrt{1 - q^{2l}} - \la_{n+1,l} \cd q^{-n} \sqrt{1 - q^{2(n+1)}} \Big) \\
& \qqqq \cd \Ga (S^H \Ga)^{m - l} (f_{n+1,l}) . 
\end{split}
\end{equation}
We now fix $m \in \nn \cup \{0\}$ and consider the first term in the above sum separately. Note that $\la_{p,-1}=0$ for all $p \in \nn \cup \{0\}$:
\[
\begin{split}
& \sum_{l = 0}^m \sum_{n = l}^\infty \la_{n,l-1} \cd q^{-l+1} \cd \sqrt{1 - q^{2l}} \cd \Ga (S^H \Ga)^{m - l}(f_{n+1,l}) \\
& \Q = \sum_{l=0}^{m-1}\sum_{n=l}^\infty \la_{n+1,l} \cd q^{-l} \sqrt{1 - q^{2(l+1)}} \cd \Ga(S^H\Ga)^{m-l-1}(f_{n+2,l+1}) \\
& \Q = \sum_{l=0}^{m-1}\sum_{n=l}^\infty \la_{n+1,l} \cd q^{-n} \sqrt{1 - q^{2(n+1)}} \cd \Ga(S^H\Ga)^{m-l}(f_{n+1,l}).
\end{split}
\]
Thus, combining this with  Expression (\ref{eq:out-vder}), we obtain the formula
\[
\begin{split}
\int^V \pa_1(x) 
& = \sum_{m = 0}^\infty \frac{q^m}{\sqrt{1 - q^{2(m+1)}}}\sum_{n = m}^\infty \la_{n+1,m} \cd q^{-n} \sqrt{1 - q^{2(n+1)}} \cd \Ga(f_{n+1,m}) \\
& = \sum_{m = 0}^\infty \sum_{n = m}^\infty \la_{n+1,m} \cd f_{n + 1,m}
\end{split}
\]
and the proof is complete.
\end{proof}

We end this section by proving a continuity result for the vertical integral:

\begin{lemma}\label{l:vertcontf}
Let $k \in \nn \cup \{0\}$. We have the estimate 
\[
\big\| (\int^V \xi )\cd \chi_{ [0,q^{2k}]}(A) \big\| \leq \frac{q^k\cd (k+1)}{(1-q)^2}\cd \|\xi\|_\infty  ,
\]
for all $\xi \in X$.
\end{lemma}
\begin{proof}
Let $\xi \in X$ be given. For each $m \in \nn \cup \{0\}$ and each $l \in \{0,1,\ldots,m\}$, we have that $\Ga (S^H \Ga)^{m - l} S^V ( \xi \cd \chi_{\{q^{2l}\}}(A)  b^2 q^{-2l} ) \in Y_m$.  Therefore, 
\[
(\int^V \xi) \cd \chi_{ [0,q^{2k}]}(A) = - \sum_{m = k}^\infty \frac{q^m(1 - q^2)}{\sqrt{ 1 - q^{2(m+1)}}} 
\cd \sum_{l = 0}^m \Ga (S^H \Ga)^{m - l} S^V\big( \xi \cd \chi_{\{q^{2l}\}}(A)\cd b^2 \cd q^{-2l} \big)  .
\]
We may thus estimate as follows:
\[
\begin{split}
& \big\|(\int^V \xi) \cd \chi_{ [0,q^{2k}]}(A) \big\| \leq  \sum_{m = k}^\infty \frac{q^m(1 - q^2)}{\sqrt{1 - q^{2(m+1)}}} (m + 1) \cd \|\xi\|_\infty \\
& \Q \leq \sum_{m = k}^\infty (m + 1) q^m  \cd  \|\xi\|_\infty \leq \frac{ q^k (k +1)}{(1-q)^2}\cd \|\xi\|_\infty  .
\end{split}
\]
This proves the lemma.
\end{proof}

\subsection{The quantum horizontal integral}
In this subsection we construct the horizontal part of the quantum integral. Recall the definition of the isometries $S^V : Y_k \to Y_k$ and $S^H : Y_k \to Y_{k+1}$, $k \in \nn \cup \{0\}$, from Equation \eqref{eq:isometries}.

For each $k \in \nn$, define the diagonal operator
\[
\De : Y_k \to Y_k \Q \De(f_{n,k}) = \fork{ccc}{ q^{k-n+1}\cd \left( \frac{1-q^{2(n-1)}}{1-q^{2k}}\right)^{1/2} \cd f_{n,k} & \T{for} & 0 < n \leq k \\ 0 &  & \T{elsewhere}}  .
\]
We remark that $\De : Y_k \to Y_k$ is bounded with operator norm satisfying the estimate 
\[
\| \De \|_\infty \leq q \Q \T{for all } k \in \nn  .
\]
Furthermore, for each $k \in \nn \cup \{0\}$, introduce the orthogonal projection
\[
P : Y_k \to Y_k \Q P(f_{n,k}) = \fork{ccc}{ f_{n,k} & \T{for} & 0 \leq n \leq k \\ 0 & \T{for} & n \geq k+ 1}  .
\]

\begin{dfn}
We define the \emph{quantum horizontal integral}
\[
\int^H : X  \to S_q^2
\]
by the formula
\[
\begin{split}
\int^H \xi & := \sum_{m = 0}^\infty q^m\frac{1- q^2}{(1-q^{2(m+1)})^{1/2}}\\
& \Q \Q   \cd \sum_{l = m + 1}^\infty (S^H)^* \big( ( S^H)^* \De \big)^{l - m -1}   P S^V 
(   \xi  \cd \chi_{\{q^{2l}\}}(A)\cd b^2 \cd q^{-2l} )  .
\end{split}
\]
\end{dfn}
\begin{lemma}
The quantum horizontal integral is a well-defined bounded operator with operator norm bounded by $(1-q^2)^{1/2}(1 - q)^{-2}$.
\end{lemma}
\begin{proof} Let $\xi \in X$. The result of the lemma follows from the estimate:
\[
\begin{split}
& \sum_{m = 0}^\infty q^m \frac{1- q^2}{(1 - q^{2(m+1)})^{1/2}} \cd \sum_{l = m + 1}^\infty \big\| (S^H)^* \big( ( S^H)^* \De \big)^{l - m -1}  P S^V 
(   \xi \chi_{\{q^{2l}\}}(A)\cd b^2 \cd q^{-2l} ) \big\| \\
& \Q \leq 
\sum_{m = 0}^\infty q^m \cd \sum_{l = m + 1}^\infty q^{l - m-1} \cd \| \xi \|_\infty \cd (1-q^2)^{1/2} \\
& \Q \leq (1 - q)^{-2} \cd \| \xi \|_\infty \cd (1-q^2)^{1/2}. \qedhere
\end{split}
\]
\end{proof}
We now illustrate how the horizontal integral will form the remaining part of the total quantum integral:

\begin{lemma}\label{l:quahorif}
Let $x = \mu \cd 1_{S_q^2} + \sum_{k = 0}^\infty \sum_{n =0}^\infty \la_{n,k} \cd f_{n,k} \in \T{Lip}_{D_q}(S_q^2)$ be given. We have the identity
\[
\int^H \pa_1(x) = \sum_{m = 0}^\infty \sum_{n = 0}^m \la_{n,m} \cd f_{n,m}  .
\]
\end{lemma}
\begin{proof}
Using Proposition \ref{p:totaldern} we compute as follows:
\begin{equation}\label{eq:out-hder}
\begin{split}
\int^H \pa_1(x) & = \sum_{m = 0}^\infty q^m\frac{1- q^2}{(1-q^{2(m+1)})^{1/2}} \\
& \Q \Q \cd \sum_{l = m + 1}^\infty (S^H)^* \big( ( S^H)^* \De \big)^{l - m -1}   P S^V 
(   \pa_1(x)  \cd \chi_{\{q^{2l}\}}(A)\cd b^2 \cd q^{-2l} ) \\
& = \sum_{m = 0}^\infty \frac{ q^m}{(1-q^{2(m+1)})^{1/2}}  \\
& \Q \Q \cd \sum_{l = m + 1}^\infty (S^H)^* \big( ( S^H)^* \De \big)^{l - m -1}   
 \Big( \sum_{n = 0}^{l-1} \Big( \la_{n,l-1} \cd q^{-l+1} \sqrt{1 - q^{2l}} \\
& \Q \Q \Q - \la_{n+1,l} \cd q^{-n} \sqrt{1 - q^{2(n+1)}} \Big) \cd f_{n+1,l} \Big). 
\end{split}
\end{equation}
We now fix $m \in \nn \cup \{0\}$ and consider part of the first term in the above sum separately. Indeed, we only look at the part of the sum where $l \geq m +2$, saving the remaining term for later. Note that $((S^H)^*)^2 f_{1,p}=0$ for all $p \in \nn \cup \{0\}$:
\[
\begin{split}
& \sum_{l = m + 2}^\infty (S^H)^* \big( ( S^H)^* \De \big)^{l - m -1}  \sum_{n = 0}^{l-1} \la_{n,l-1} \cd q^{-l+1} \sqrt{1 - q^{2l}}  \cd f_{n+1,l}\\
& \Q =\sum_{l = m + 1}^\infty (S^H)^* \big( ( S^H)^* \De \big)^{l - m }  \sum_{n = 0}^{l-1}  \la_{n+1,l} \cd q^{-l} \sqrt{1 - q^{2(l+1)}}  \cd f_{n+2,l+1}\\
&  \Q =\sum_{l = m + 1}^\infty (S^H)^* \big( ( S^H)^* \De \big)^{l - m -1}  \sum_{n = 0}^{l-1}  \la_{n+1,l} \cd q^{-n}\sqrt{1-q^{2(n+1)}} \cd f_{n+1,l}.
\end{split}
\]
Thus, combining with Expression (\ref{eq:out-hder}), we obtain the formula
\[ 
\begin{split}
\int^H \pa_1(x) &= \sum_{m = 0}^\infty \frac{q^m}{(1-q^{2(m+1)})^{1/2}} (S^H)^* \sum_{n=0}^m \la_{n,m} \cd q^{-m}\sqrt{1-q^{2(m+1)}}f_{n+1,m+1}\\
& =\sum_{m = 0}^\infty \sum_{n=0}^m \la_{n,m} \cd f_{n,m}
\end{split}
\]
and the proof is complete.
\end{proof}

We end this subsection by providing a continuity result for the horizontal integral:

\begin{lemma}\label{l:horicont}
Let $k \in \nn\cup \{0\}$. We have the estimate:
\[
\big\| (\int^H \xi) \cd \chi_{[0,q^{2k}]}(A) \big\| \leq q^k \cd (1 - q)^{-2} \cd \| \xi \|_\infty  \Q \mbox{for all }\xi \in X.
\]
\end{lemma}
\begin{proof}
This follows from the estimate:
\[
\begin{split}
& \sum_{m = k}^\infty q^m\frac{1- q^2}{(1-q^{2(m+1)})^{1/2}}  \cd \sum_{l = m + 1}^\infty \big\| (S^H)^* \big( ( S^H)^* \De \big)^{l - m -1}  P S^V 
(   \xi  \chi_{\{q^{2l}\}}(A)\cd b^2q^{-2l} ) \big\| \\
& \Q \leq \sum_{m = k}^\infty q^m \cd \sum_{l = m+1}^\infty q^{l - m - 1} \cd \| \xi \|_\infty 
= q^k \cd (1 - q)^{-2} \cd \| \xi \|_\infty  . \qedhere
\end{split}
\]
\end{proof}

\subsection{Proofs of main results}\label{ss:proofs}
We end this section by defining the quantum integral and proving our main results about it.

\begin{dfn}
We define the \emph{quantum integral} as the sum of the quantum vertical integral and the quantum horizontal integral:
\[
\int := \int^V + \int^H : X \to S_q^2  .
\]
\end{dfn}

\begin{proof}[Proof of Theorem \ref{t:fundam}]
This is an immediate consequence of Lemma \ref{l:quavertf} and Lemma \ref{l:quahorif}.
\end{proof}

\begin{proof}[Proof of Proposition \ref{p:continuity}]
This follows immediately from Lemma \ref{l:vertcontf} and Lemma \ref{l:horicont}.
\end{proof}

\section{The Podle\'s sphere as a spectral metric space}
We are now ready to prove the main result of this paper. We recall that the seminorm $L_{D_q} : S_q^2 \to [0,\infty]$ is coming from the D\polhk{a}browski-Sitarz spectral triple $( \C O(S_q^2), H_+ \op H_-, D_q)$ by taking commutators with the selfadjoint unbounded operator $D_q : \T{Dom}(D_q) \to H_+ \op H_-$. In fact, for any $x$ in the Lipschitz algebra $\T{Lip}_{D_q}(S_q^2) \su S_q^2$ we have the identity
\[
L_{D_q}(x) = \max\{ \| \pa_1(x) \|_\infty , \| \pa_1(x^*) \|_\infty \},
\]
see Section \ref{s:prelim} and Section \ref{s:spectral}. We are going to prove that $(S_q^2, L_{D_q})$ is a compact quantum metric space. The crucial step in this direction is contained in the following:

\begin{prop}\label{p:l-compact}
The subset
\[
\sL := \big\{ x \in S_q^2 \mid L_{D_q}(x) \leq 1 \, , \, \, \psi_\infty(x) = 0 \big\} \subseteq S_q^2
\]
is totally bounded.
\end{prop}
\begin{proof}
Remark first that Theorem \ref{t:fundam} and Proposition \ref{p:continuity} implies that
\[
\| x \| = \| \int \pa_1(x) \| \leq \| \pa_1(x) \|_\infty \cd 2 \cd (1 - q)^{-2} \leq 2 \cd (1 - q)^{-2}
\]
for all $x \in \sL$. In particular, $\sL \subseteq S_q^2$ is a bounded subset. 

For each $j \in \nn \cup \{0\}$ define the constant
\[
C_j := 1 + 2 \cd (1 - q)^{-2} \cd  \big\| \pa_1( \chi_{\{q^{2j}\}}(A)) \big\| 
\]
and notice that
\[
\| \pa_1(x \cd \chi_{\{q^{2j}\}}(A)) \| \leq \| \pa_1(x) \|_\infty + \| x \| \cd \big\| \pa_1( \chi_{\{q^{2j}\}}(A)) \big\|
\leq C_j
\]
for all $x \in \sL$.

Let now $\ep > 0$ be given. By Proposition \ref{p:continuity} we may find a $k \in \nn$ such that
\[
\| x \cd \chi_{[0,q^{2k}]}(A) \| = \| (\int \pa_1(x) ) \cd \chi_{[0,q^{2k}]}(A) \|< \ep/2  , 
\]
for all $x \in \sL$.

By Theorem \ref{t:0-totbou} we may, for each $j \in \{0,\ldots,k-1\}$, find a finite number of elements $x_{j,1},\ldots, x_{j,n_j} \in S_q^2$ such that
\[
\big\{ y \in Y_j \cap \T{Lip}_{D_q}(S_q^2) \mid \| \pa_1(y) \| \leq 1 \big\}  \subseteq \bigcup_{i = 1}^{n_j} \B B_{\frac{\ep}{2 \cd C_j \cd k}} (x_{j,i}) ,
\]
where $\B B_r(z)$ denotes the norm-ball with radius $r > 0$ and center $z \in S_q^2$.

Define the finite subset
\[
S := \big\{ \sum_{j = 0}^{k-1} x_{j,i_j} \cd C_j \mid (i_0,\ldots,i_{k-1}) \in \{1,\ldots,n_0\} \ti \ldots \ti \{1,\ldots, n_{k-1}\} \big\}
\subseteq S_q^2  .
\]
We claim that
\begin{equation}\label{eq:inclubou}
\sL \subseteq \cup_{y \in S} \B B_\ep(y)  .
\end{equation}

Thus, let $x \in \sL$ be given. For each $j \in \{0,\ldots,k-1\}$ we may choose $i_j \in \{1,\ldots,n_j\}$ such that
\[
x \cd \chi_{\{q^{2j}\}}(A) \cd C_j^{-1} \in \B B_{\frac{\ep}{2 \cd C_j \cd k}} (x_{j,i_j})  .
\]
We thus have that
\[
\begin{split}
\| x - \sum_{j = 0}^{k-1} x_{j,i_j} \cd C_j \| 
& \leq \| x \cd \chi_{[0,q^{2k}]}(A) \| + \sum_{j = 0}^{k-1} \| x \cd \chi_{\{q^{2j}\}}(A)- x_{j,i_j} \cd C_j \| \\
& < \frac{\ep}{2} + k \cd \frac{\ep}{2 \cd k} = \ep  .
\end{split}
\]
This proves the inclusion in Equation \eqref{eq:inclubou} and hence the result of the theorem.
\end{proof}

\begin{prop}\label{p:l-kernel}
Let $x \in S_q^2$. We have the biimplication
\[
L_{D_q}(x) = 0 \lrar x = \psi_\infty(x) \cd 1_{S_q}  .
\]
\end{prop}
\begin{proof}
This is a consequence of Theorem \ref{t:fundam}.
\end{proof}


\begin{thm}
The pair $(S_q^2,L_{D_q})$ is a compact spectral metric space.
\end{thm}
\begin{proof}
A combination of Proposition \ref{p:l-kernel}, Proposition \ref{p:l-compact}, Theorem \ref{t:Rieffel} and Theorem \ref{t:dabsit} completes the proof.
\end{proof}

\bibliographystyle{plain}

\end{document}